\documentclass[12pt]{amsart}
\usepackage{amsmath}
\usepackage{amssymb}
\usepackage{epsfig}
\usepackage{graphicx}
\numberwithin{equation}{section}

\input xy
\xyoption{all}

\calclayout
\allowdisplaybreaks[3]

\theoremstyle{plain}
\newtheorem{prop}{Proposition}

\newtheorem{theo}[prop]{Theorem}
\newtheorem{coro}[prop]{Corollary}

\newtheorem{lemm}[prop]{Lemma}

\theoremstyle{definition}

\newtheorem{rema}[prop]{Remark}

\newtheorem{exam}[prop]{Example}

\makeatother
\makeatletter

\newcommand{\bC}{\mathbb C}

\newcommand{\bG}{\mathbb G}

\newcommand{\bP}{\mathbb P}
\newcommand{\bQ}{\mathbb Q}

\newcommand{\bZ}{\mathbb Z}

\newcommand{\cE}{\mathcal E}

\newcommand{\cP}{\mathcal P}

\newcommand{\cX}{\mathcal X}
\newcommand{\cY}{\mathcal Y}

\newcommand{\rH}{\mathrm H}

\newcommand{\Bl}{\operatorname{Bl}}
\newcommand{\Br}{\operatorname{Br}}

\newcommand{\Pic}{\operatorname{Pic}}
\newcommand{\Spec}{\operatorname{Spec}}
\newcommand{\Aut}{\operatorname{Aut}}

\newcommand{\Hom}{\operatorname{Hom}}

\newcommand{\FM}{\mathsf{FM}}

\newcommand{\ra}{\rightarrow}

\author{Brendan Hassett}
\address{Department of Mathematics\\
Brown University
Box 1917 \\
151 Thayer Street,
Providence, RI 02912 \\
USA}
\email{brendan\underline{ }hassett@brown.edu}

\author{Yuri Tschinkel}
\address{Courant Institute\\
                New York University \\
                New York, NY 10012 \\
                USA }
\email{tschinkel@cims.nyu.edu}

\address{Simons Foundation\\
160 Fifth Avenue\\
New York, NY 10010\\
USA}

\title[Equivariant derived equivalence and rational points]{Equivariant derived equivalence and rational points on K3 surfaces}

\begin{document}
\date{February 9, 2023}

\begin{abstract}
We study arithmetic properties of derived equivalent 
K3 surfaces over the field of Laurent power series, 
using the equivariant geometry of K3 surfaces with 
cyclic groups actions.  
\end{abstract}

\maketitle

\section{Introduction}
\label{sect:intro}

Let $X$ and $Y$ be smooth K3 surfaces over a nonclosed field $K$. 
Suppose $X$ and $Y$ are derived equivalent over $K$, i.e., 
there is an equivalence
of bounded derived categories 
of coherent sheaves
$$
\Phi: D^b(X)\to D^b(Y),
$$
as triangulated categories, defined over $K$. 
Such a derived equivalence 
respects (see \cite[Section 1]{HT-derived}):
\begin{itemize}
\item
the Galois action on geometric Picard groups, 
\item the Brauer groups, 
\item the {\em index}, i.e., the $\gcd$ of degrees of field extensions $K'/K$ such that
$X(K')\neq \emptyset$. 
\end{itemize}

We are interested in understanding which other {\em arithmetic} properties are preserved under $\Phi$.  
Specifically, in \cite{HT-derived} we asked whether or not
$$X(K) \neq \emptyset \quad  \Leftrightarrow
\quad Y(K) \neq \emptyset.
$$
This is known when
\begin{itemize}
\item 
$K=\mathbb F_q$ is finite, $\mathrm{char}(K)>2$,  \cite{LO}, \cite[16.4.3]{Hu},
\item 
$K$ is real \cite[Prop. 25]{HT-derived},
\item $K=\mathbb C((t))$ \cite[Cor. 30]{HT-derived}, assuming that local monodromy has trace $\neq -2$, in which case both $X(K), Y(K)\neq \emptyset$,
\item $K$ is $p$-adic, under strong assumptions on the reduction and for $p\ge 7$ \cite[Prop. 36]{HT-derived}.
\end{itemize}

We propose to study this in a very special case -- isotrivial 
families of K3 surfaces over the punctured disc.  Let $G=C_N$
be a finite cyclic group of order $N$. Fix projective
K3 surfaces $X$ and $Y$ over $\bC$ with $G$-actions and consider
the associated isotrivial families
$$
\cX,\cY \rightarrow \Delta_1:=\Spec(\bC((t))),
$$
with generic fibers $\cX_t$ and $\cY_t$ over $K=\bC((t))$, as defined in Section~\ref{subsect:RFP}.

\begin{theo}\label{theo:main}
Suppose that $\cX_t$ and $\cY_t$ admit a derived equivalence
$$
\Phi: D^b(\cX_t) \simeq D^b(\cY_t),
$$
over $K$. 
If $\cX_t(K)\neq \emptyset$ then $\cY_t(K) \neq \emptyset$.
\end{theo}

Related questions were considered by \cite{AAFH} (hyperk\"ahler fourfolds) and 
twisted K3 surfaces \cite{ADPZ}; here the existence of rational points is {\em not} compatible with derived equivalence. The case of torsors for abelian varieties is addressed in \cite{AKW}.

\

Our approach is based on the analogy between equivariant geometry and 
descent for nonclosed fields. Section~\ref{sect:gen} presents foundations
for derived equivalence in the presence of group actions, with a view toward
equivariant approaches to the Mukai lattice. We link isotrivial families over
fields of Laurent series to equivariant geometry in Section~\ref{sect:isotrivial}.
Section~\ref{sect:fixed} presents the proof of Theorem~\ref{theo:main} through analysis of fixed
points; we close with a discussion of connections
with the Burnside formalism and open questions.

\

\noindent
{\bf Acknowledgments:} 
The first author was partially supported by Simons Foundation Award 546235 and NSF grant 1701659,
the second author by NSF grant 
2000099. We are grateful to Andrew Kresch for his help towards a correct formulation of equivariant criteria, to Nicolas Addington for his comments on descending
equivariant equivalences, and to Barry Mazur
for his suggestion to find examples along the lines of Example~\ref{exam:Mazur}.

\section{Generalities}
\label{sect:gen}

\subsection{Equivariant derived equivalence} 
\label{subsect:EDE}
We follow \cite{ploog} and refer the reader to \cite{KS} for a more general
approach.  

Let $k$ be an algebraically closed field of characteristic zero and
$X$ a smooth projective variety over $k$ equipped with the action of a finite group $G$.
We consider the bounded derived category $D^b(X,G)$ of $G$-equivariant complexes of coherent sheaves on $X$, i.e., objects are pairs $\cP=(P,\rho)$ consisting of complexes $P$ of coherent sheaves and $G$-linearizations
$\rho$ compatible with differentials \cite{ploog}. This is compatible with intrinsic formulations of
$G$-actions on triangulated categories \cite[\S 9]{elagin}, under our assumptions.  

Suppose that $X$ and $Y$ are smooth projective varieties with $G$-actions. Given an 
element 
$$
\cP=(P,\rho) \in D^b(X\times Y,G\times G)
$$
there is an equivariant Fourier-Mukai transform
$$
\FM_{\cP}(-,G): D^b(X,G) \rightarrow D^b(Y,G),
$$
obtained by pulling back via projection to $X$, tensoring by $\cP$, and pushing forward via
projection to $Y$ \cite[\S~1.2]{ploog}. This operation makes sense \cite[Lemma~5]{ploog}
provided $\cP$ is equivariant for the diagonal $G_{\Delta} \subset G\times G$ only, and the 
equivariant Fourier-Mukai transform is compatible with the ordinary Fourier-Mukai 
transform associated with $P$. (In other words, we can forget the $G$-actions.) 
Furthermore, if $P$ induces an equivalence of ordinary derived categories then 
$\cP$ induces an equivalence of the equivariant derived categories.  

We assume that $G$ acts faithfully on $X$ and $Y$.  
Conversely, suppose that $P \in D^b(X\times Y)$ induces an equivalence. When can it be lifted 
to an equivariant derived equivalence? By \cite[Lem.~4]{ploog}, each kernel $P$ inducing an equivalence must be simple, i.e., every automorphism of $P$ as an element of the derived category may be represented
as rescaling of a representative complex. In particular,
if $P$ is $G$-invariant as an element of the derived
category -- $(g,g)^*P$ is quasi-isomorphic to $P$ 
for each $g\in G$ -- then the underlying complex of sheaves is
$G$-invariant.  Using the identification
$\Aut(P)=\bG_m$, for our desired lifting it is necessary that the resulting cocycle $\alpha \in \rH^2(G,\bG_m)$  \cite[Lem.~1]{ploog} vanish. 
When $G$ is cyclic, $\rH^2(G,\bG_m)=0$ and $\alpha$ vanishes automatically.

If $P$ does lift to an equivariant complex $\cP=(P,\rho)$ then this typically is not unique.
We can tensor $\rho$ freely with any character of $G$.   

\subsection{Specialization to K3 surfaces}
We retain the notation of Section~\ref{subsect:EDE} and assume that $X$ and $Y$ are K3 surfaces with Mukai lattices
$\widetilde{\rH}(X,\bZ)$ and $\widetilde{\rH}(Y,\bZ)$.
Suppose that $X$ and $Y$ 
are derived equivalent, with the equivalence realized
by an isomorphism
$$
i:Y\stackrel{\sim}{\longrightarrow} M_v(X),
$$ 
where 
$$v=(r,D,S) \in \rH^0(X,\bZ) \oplus \rH^2(X,\bZ) \oplus
\rH^4(X,\bZ)$$
is the Mukai vector of a moduli space of vector
bundles. See \cite[Prop.~10.10]{Hu-FM} for more details;
in particular, since $v$ induces
a derived equivalence, $r$ and $s$ are relatively prime
and we may assume $r>0$.  
The kernel $P \in D^b(X\times Y)$ inducing the equivalence may be interpreted as a universal sheaf over $X \times M_v(X)$.
We have suppressed the polarization from the notation because it is irrelevant for our analysis;
under our assumptions, any ample line bundle will yield a fine moduli space parametrizing stable sheaves \cite[Prop.~10.20]{Hu-FM}.

Suppose now that $X$ and $Y$ come with faithful actions by a finite group $G$, where $v$ is $G$-invariant so that $M_v(X)$
admits a $G$-action. Here, we are implicitly using a
$G$-invariant polarization so stability is compatible
with the $G$ action.

Fix an equivariant isomorphism $i:Y\stackrel{\sim}{\ra}M_v(X)$ as above.
This is not sufficient to produce an equivariant derived equivalence between
$X$ and $Y$.  The issue is the existence of an {\em equivariant} universal sheaf 
$E \rightarrow X \times M_v(X)$. Given an arbitrary universal sheaf $E$,
simplicity of the sheaves parametrized by $M_v(X)$ yields
$$
g^*E \simeq E \otimes p_2^*L_g,\quad  g\in G,
$$
where $L_g$ is a line bundle on $M_v(X)$. The data $(L_g)_{g\in G}$ defines an element in $\rH^1(G,\Pic(M_v(X)))$. Assuming this vanishes, we can produce
an invariant kernel $P$ on $X\times Y$. As we have seen, the obstruction to
lifting $P$ to an equivariant complex $\cP$ then lies in $\rH^2(G,\bG_m)$.   

Both these obstructions are encoded by
$$
\ker\big(\Br(M_v(X),G) \rightarrow \Br(M_v(X)) \big)
$$
in the equivariant Brauer group, computed by a spectral sequence with $\mathrm E_2$-terms \cite[\S~2.3]{HT21} 
$$
\rH^2(G,\bG_m)\, \text{ and }\, \rH^1(G,\Pic(M_v(X))).
$$
Ploog's cocycle $\alpha$ lies in the kernel of the natural arrow
$$\rH^2(G,\bG_m) \rightarrow \Br(M_v(X),G)$$
induced by the structure map of $M_v(X)$. 
This vanishes when $M_v(X)$ admits a fixed point.  

Mukai \cite{Mukai} and Orlov \cite[Th.~3.3]{Orlov} have shown that
K3 surfaces 
$X$ and $Y$ are
derived equivalent if and only if there is an isomorphism of transcendental lattices
$$
T(X) \simeq T(Y),
$$
as Hodge structures.  This does not suffice in the equivariant case:

\begin{prop} 
\label{prop:Mukaicritnew}
Let $X$ and $Y$ be complex projective K3 surfaces with  
faithful actions by a finite group $G$.
Then we have a sequence of implications:
\begin{enumerate}
\item{there is a $G$-equivariant derived equivalence $D^b(X) \simeq D^b(Y)$;}
\item{there is an isomorphism of Mukai lattices
$$
\widetilde{\rH}^*(X,\bZ) \simeq \widetilde{\rH}^*(Y,\bZ)
$$
respecting the Hodge structures and the $G$-actions;} 
\item{there is a $G$-equivariant isomorphism
$$T(X) \simeq T(Y)$$
of transcendental lattices,
compatible with Hodge structures.}
\end{enumerate}
\end{prop}
\begin{proof}
Suppose that $X$ and $Y$ are equivariantly derived equivalent. Then
there is an isomorphism $i:Y\simeq M_v(X)$ such that the universal sheaf
$$
E \rightarrow X \times M_v(X)
$$
admits a $G$-linearization $\rho$ such that
$\FM_{(E,\varrho)}$ is an equivalence.  
The cohomological Fourier-Mukai transform and $i$ induce an isomorphism
$$
i^* \circ \FM_E:\widetilde{\rH}^*(X,\bZ) \rightarrow \widetilde{\rH}^*(Y,\bZ)$$
taking $v$ to $(0,0,1)$. 
The homomorphism $i^*\circ \FM_E$ induces the desired isomorphism of transcendental cohomology groups.
\end{proof}

Reversing the first implication in Proposition~\ref{prop:Mukaicritnew} is not possible precisely when the obstruction $\alpha \in \rH^2(G,\bG_m)$ is nonzero. Since the obstruction $\alpha$ vanishes in the cyclic case we have:

\begin{coro} \label{coro:Mukaicrit}
Suppose that $X$ and $Y$ are complex projective K3 surfaces with faithful actions
by a cyclic group $G$. Then there is a $G$-equivariant derived equivalence between them iff there is an isomorphism of their Mukai lattices
respecting the Hodge structures and the $G$-actions.
\end{coro}

\begin{rema}
\label{rema:crit}
The second implication in Proposition~\ref{prop:Mukaicritnew}
also fails to be an equivalence in general. To extend an
isomorphism $T(X) \simeq T(Y)$ to an isomorphism of Mukai lattices, we require a $G$-equivariant 
isomorphism of lattices
$$\Pic(X) \stackrel{\sim}{\longrightarrow} \Pic(Y)$$
compatible (on discriminant groups) with the given isomorphism of transcendental lattices. 
By definition, $T(X)$ is the orthogonal complement to $\Pic(X)$ in $\rH^2(X,\bZ)$.  
Example~\ref{exam:compatible} shows such a homomorphism might not exist.  
\end{rema}

\begin{exam} 
\label{exam:compatible}
Given a polarized K3 surface of degree two $(X,f), f^2=2$, the linear
series $|f|$ induces a double cover $X \rightarrow \bP^2$ \cite[Th.~3.1 and Prop.~8.1]{SD}, branched over a smooth
plane curve of degree six.  
The covering involution $\iota$ acts on $f^{\perp} \subset \rH^2(X,\bZ)$
by multiplication by $-1$.

Let $X$ be a K3 surface surface with 
$$\Pic(X) = \begin{array}{r|cc}  & f_1 & f_2 \\
                            \hline
                            f_1 & 2 & 5 \\
                            f_2 & 5 & 2 
                            \end{array}, $$
with involutions $\iota_1$ and $\iota_2$ associated with the double covers $X\rightarrow \bP^2$ induced by $f_1$ and $f_2$.  
Each involution acts on the primitive cohomology -- hence the 
transcendental cohomology $T(X)$ -- by $-1$. 
However, we shall show there is no automorphism of the Mukai lattice  
$$
a:\widetilde{\rH}(X,\bZ) \rightarrow \widetilde{\rH}(X,\bZ)
$$
compatible with Hodge structures and conjugating these involutions. In particular (3) does not imply (2) in
Proposition~\ref{prop:Mukaicritnew}.

We argue by contradiction; assume such an $a$ existed. We have 
$$
\iota_1(2f_2-5f_1)=-(2f_2-5f_1) \quad 
\iota_2(2f_1-5f_2)=-(2f_1-5f_2),
$$
the unique (up to sign) elements of the Mukai lattice that are
algebraic with eigenvalue $-1$.
Thus we must have
$$a(2f_2-5f_1) = \pm (2f_1-5f_2).$$

The discriminant group $d(\Pic(X))=\Hom(\Pic(X),\bZ)/\Pic(X)$ is 
$$
\bZ/21\bZ\simeq \bZ/3\bZ \times \bZ/7\bZ,
$$ 
with generators $d_1=\frac{f_1-f_2}{3}$ and $d_2= \frac{f_1+f_2}{7}$. 
Our distinguished elements give generators 
$$\frac{2f_2-5f_1}{21}=-d_1+3d_2 \quad
\frac{2f_1-5f_2}{21}=d_1-4d_2.$$
Note that these are not equal, even up to sign. 
We conclude that any automorphism of the algebraic classes
$$\Pic(X) \oplus \rH^0(X,\bZ) \oplus \rH^4(X,\bZ) \subset 
\widetilde{\rH}(X,\bZ)$$
conjugating $\iota_1$ and $\iota_2$ acts on the discriminant
group by an element $\neq \pm 1$. In particular, this applies to
$$a|_{\Pic(X) \oplus \rH^0(X,\bZ) \oplus \rH^4(X,\bZ)}$$

The only automorphisms of the transcendental cohomology $T(X)$ -- 
assuming $X$ is general with the stipulated Picard group
-- are multiplications by $\pm 1$.  These are the only
elements commuting with the action of the Hodge
group of a general such $X$, which is the identity
component of the orthogonal group associated with the intersection form.  Thus 
$$a|_{T(X)} = \pm 1$$
and the same holds true on the discriminant group.
This gives a contradiction: Nikulin's theory gives an isomorphism
$$d(T(X)) \simeq d(\Pic(X))$$
and any automorphism of the full cohomology (compatible with the
Hodge decomposition)
must respect this isomorphism.
\end{exam}

Remark~\ref{rema:crit} is reminiscent of \cite[Exam.~4.11]{HS}:
Isomorphisms of transcendental cohomology groups of {\em twisted} K3 surfaces need not lift to 
twisted derived equivalences. 

\

We close with examples of intriguing derived equivalences relating K3 surfaces with involution:
\begin{exam} \label{exam:Mazur}
Recall that the derived category of any smooth projective variety $X$ has an involution
$$\begin{array}{rcl}
i_X: D^b(X) & \rightarrow& D^b(X) \\
      \cE & \mapsto& (\cE[1])^{\vee}
      \end{array},$$
i.e., the composition of ``shift-by-one'' and ``taking duals''.
When $X$ is a K3 surface, $i_X$ acts on $\widetilde{\rH}(X,\bZ)$
by the identity on $\rH^2$ and multiplication
by $-1$ on $\rH^0$ and $\rH^4$. Note that $i_X$ is {\em not} an 
autoequivalence -- indeed it fails to be orientation-preserving,
a necessary condition for autoequivalences \cite[\S 4]{HMS}. 

We seek degree two K3 surfaces $(X,f)$ and $(Y,g)$
(cf.~Example~\ref{exam:compatible}) with associated involutions
$$\iota:X \rightarrow X, \quad \kappa:Y \rightarrow Y,$$
such that $(D^b(X),i_X\circ \iota)$ and $(D^b(Y),i_Y \circ \kappa)$ are $C_2$-equivariantly 
derived equivalent but $(X,f)$ and $(Y,h)$ are not isomorphic. 
Analogous to  
Corollary~\ref{coro:Mukaicrit}, we would like  
equivariant isomorphisms of Mukai lattices (with Hodge structures)
$$a:\widetilde{\rH}(X,\bZ) \simeq \widetilde{\rH}(Y,\bZ)$$
where there is no equivariant isomorphism
$$\rH^2(X,\bZ) \not \simeq \rH^2(Y,\bZ).$$

These may be produced using the theory of binary quadratic forms
\cite{Buell}. Consider even, negative definite, rank-two lattices represented 
by symmetric integer matrices $A$ and $B$. We say that they are in
the same {\em genus} if they are $p$-adically equivalent for all primes $p$;
this is equivalent \cite[Cor.~1.13.4]{nik-lattice} to stable equivalence
$$A \oplus U \simeq B \oplus U, \quad 
U = \left(\begin{matrix} 0 & 1 \\ 1 & 0 \end{matrix} \right).$$
There are criteria, expressed via class groups, for the existence of
non-isomorphic lattices in the same genus; see \cite[App.~1]{Buell} for 
tables. 

We seek examples of such lattices $A$ and $B$, subject to the condition that 
$A$ and $B$ do not represent $-2$.  This last assumption ensures that the divisors $f$ and $g$
are ample. For instance, consider even positive definite binary forms of discriminant $-47$;
the reduced forms are:
$$\left( \begin{matrix} 2 & 1 \\ 1 & 24 \end{matrix} \right), \
\left( \begin{matrix} 4 & 1 \\ 1 & 12 \end{matrix} \right), \
\left( \begin{matrix} 4 & -1 \\ -1 & 12 \end{matrix} \right), \
\left( \begin{matrix} 6 & 1 \\ 1 & 8 \end{matrix} \right), \
\left( \begin{matrix} 6 & -1 \\ -1 & 8 \end{matrix} \right).
$$
Only the first of these represents $2$ so we could take
$$A= - \left( \begin{matrix} 4 & 1 \\ 1 & 12 \end{matrix} \right), \ 
B=-\left( \begin{matrix} 6 & 1 \\ 1 & 8 \end{matrix} \right).$$

We construct the desired K3 surfaces using surjectivity of the Torelli map.
Choose a K3 surface $X$ with 
$$\Pic(X) = \bZ f \oplus A$$
with involution $\iota$ fixing $f$ and acting on $A$ and $T(X)$ via $-1$.  
There exists a second K3 surface $Y$
with 
$$\Pic(Y) = \bZ g \oplus B$$
and $T(X) \simeq T(Y)$.  
This admits an involution $\kappa$ acting on $B$ and $T(Y)$ via $-1$. 
There is no isomorphism $\Pic(X) \simeq \Pic(Y)$ compatible with the involutions.
However the stable equivalence of $A$ and $B$ induces
$$\widetilde{\rH}(X,\bZ) \simeq U \oplus \rH^2(X,\bZ) \simeq U \oplus \rH^2(Y,\bZ)
\simeq \widetilde{\rH}(X,\bZ),$$
compatible with $\iota$ and $\kappa$.  
The involutions act on the $U$ summands via multiplication by $-1$.
\end{exam}

We will explore this further in \cite{HT23}.

\section{Isotrivial families}
\label{sect:isotrivial}

\subsection{Construction}
\label{subsect:construct}
Let $X$ be a projective
K3 surface and 
$$
G=C_N \subseteq \Aut(X)
$$ 
a finite cyclic subgroup of the automorphism group of $X$. 
Let $\Delta_2=\Spec(\bC[[\tau]])$ be a formal disc on which 
$G$ acts via
$$\tau \mapsto \zeta \tau, \quad \zeta=\exp(2\pi i/N).$$ 
The $G$-equivariant projection 
$$X\times \Delta_2 \ra \Delta_2$$
induces an isotrivial family
$$\pi:\cX:=(X\times \Delta_2)/G \ra \Delta_1:=\Delta_2 /G.$$

Let $K=\bC((t))$ and $L=\bC((\tau))$ denote the fields associated with $\Delta_1$ and
$\Delta_2$. We regard $\cX_t$ as a K3 surface over $K$; a $K$-rational point of $\cX_t$
is equivalent to a section of $\pi$.

\begin{prop}
Suppose that $X$ and $Y$ are complex K3 surfaces with faithful actions of $G=C_N$; assume they
are $G$-equivariantly derived equivalent.  Then $\cX_t$ and $\cY_t$ are
derived equivalent over $K$.
\end{prop}
Actually, our proof will give more: It suffices to assume that there exists a 
$G$-invariant complex $P$ inducing the equivalence between $X$ and $Y$ (see
Section~\ref{sect:gen}).
\begin{proof}
Realize 
$$
i:Y \stackrel{\sim}{\longrightarrow} M_v(X)
$$ 
for some
Mukai vector $v$ for $X$, fixed under the $G$-action. 
This isomorphism may be chosen to be equivariant under the $G$-action.
Letting $\tau=\sqrt[N]{t}$, we basechange to an isomorphism
$$\cY_{\tau} \simeq M_v(\cX_{\tau}).$$ 
This descends to an isomorphism 
$$\cY_t \simeq M_v(\cX_t),$$
where the latter is the coarse moduli space.  To complete
the proof, we need that $M_v(\cX_t) \times \cX_t$ admits a universal sheaf.  
Since the underlying sheaves are simple, this universal sheaf is unique up to tensoring by line 
bundles on $M_v(\cX_t)$ -- a trivial line bundle given our assumption
that $P$ is $G$-invariant. Thus the obstruction to descending the data associated with
$P$ to a sheaf defined over $K$ lives in the Brauer group of $K$.
The triviality of $\Br\big(\bC((t))\big)$ shows this obstruction vanishes.
\end{proof}

\subsection{Rational points and fixed points}
\label{subsect:RFP}

\begin{prop}
\label{prop:section}
The morphism
$$
\pi:\cX\to \Delta_1
$$ 
admits a section if and only if the action of $G$ on
$X$ admits a fixed point.
\end{prop}
\begin{proof}
If $\pi$ admits
a section $\sigma_1:\Delta_1 \ra \cX$ then the induced section
$\sigma_2:\Delta_2 \ra \cX\times_{\Delta_1}\Delta_2$ is $G$-invariant, whence $\sigma_2(0)$
is fixed.  

Suppose $X$ has a fixed point. Then the resulting constant
section of $X\times \Delta_2 \rightarrow \Delta_2$ is 
invariant under the action of $G$ and thus descends to
a section of $(X\times \Delta_2)/G \rightarrow \Delta_2/G$.
\end{proof}

\section{Fixed point analysis}
\label{sect:fixed}

Let $X$ be a K3 surface over an algebraically closed field of characteristic zero and $\sigma\in\Aut(X)$ an 
automorphism of order $N$. In the following sections, we analyze the structure 
of the fixed point locus 
$$
X^{\sigma}:=\{ x\in X \,|\, \sigma(x)=x\},
$$
with the goal of identifying $\sigma$ such that $X^{\sigma}=\emptyset$. 

\subsection{Cyclic automorphisms}
\label{sect:cyclic}

We review basic properties of finite automorphisms due to Nikulin \cite{NikFinite}.
Suppose that $G=\left< \sigma \right>= C_N$ acts on a K3 surface $X$.  
We have an exact sequence
\begin{equation} \label{eqn:seq}
0 \ra  C_n  \ra G \ra  \mu_m \ra 0, \quad nm=N,
\end{equation}
where $C_n$ 
is the kernel of the representation of $G$ on the symplectic form.
Elements in $C_n$ are called {\em symplectic}; when $C_n=1$, the action is called {\em purely nonsymplectic}. We write $N=n\cdot m$, to emphasize the symplectic versus nonsymplectic actions.   

\begin{prop}\label{prop:equal}
Let $X_1$ and $X_2$ derived equivalent K3 surfaces.
Assume that both carry a faithful action of $G=C_N$ and that the derived equivalence is compatible with $G$. Then the factorizations
$$
N=n_1m_1 = n_2m_2,
$$
encoding the symplectic elements, are equal, i.e., 
$$
n_1=n_2\quad \text{ and }\quad m_1=m_2.
$$
\end{prop}
\begin{proof}
We can read off the symplectic automorphisms from the action on the Mukai lattice, as the symplectic form is
distinguished in its complexification.  
\end{proof}

\subsection{Fixed point formulas}
\label{sect:fix}

Let $G=\langle \sigma\rangle$ be a cyclic group acting on a K3 surface $X$. Let 
$$
\sigma^*:\tilde{\rH}(X,\bZ) \to \tilde{\rH}(X,\bZ)
$$
be the induced action on the Mukai lattice, and 
$$
\chi(\sigma) :=\mathrm{Tr}(\sigma^*)
$$
the corresponding trace.

The {\em topological} fixed point formula
takes the form:
\begin{equation}
\label{eqn:top}
\chi( X^\sigma) = \chi(\sigma),
\end{equation}

Since $\chi(\sigma)$ may
be read off from the action on the Mukai lattice, $\chi(X^\sigma)$ is an invariant of $G$-equivariant
derived equivalence.

\begin{lemm}
\label{lemm:symf}
Let $N=n\cdot m$ with $n\ge 2$.
Then $X^\sigma$ is empty or a finite set of isolated points, and
$$
\chi(X^{\sigma}) = \# X^\sigma.
$$
\end{lemm}

\begin{proof}
By \cite{NikFinite}, symplectic automorphisms do not contain curves in their fixed locus (a detailed description of possible $X^\sigma$ is in Section~\ref{sect:RCP}). 
\end{proof}

The {\em complex} Lefschetz fixed point formula involves sums 
\begin{equation}
\label{eqn:hol2}
\sum_{\mathfrak p} a(\mathfrak p) \, +\, \sum_{C} b(C),
\end{equation}
of contributions from fixed points and fixed curves; here $\zeta=\zeta_N$ (see, \cite[p. 567]{At-S-3}). 
The corresponding contributions are given by
$$
a(\mathfrak p)= \frac{1}{(1-\zeta^i)(1-\zeta^j)},
$$
for fixed points $\mathfrak p$ with weights $\beta_{\mathfrak p}=(i,j)$ in 
the tangent bundle 
at $\mathfrak p$, and 
$$
b(C) =\frac{1-g(C)}{1-\zeta^{-r(C)}} - 
\frac{\zeta^{-r(C)}}{(1-\zeta^{-r(C)})^2} C^2,
$$
where $g(C)$ is the genus of $C$, and $r(C)$ is the weight in the normal bundle to $C$.
For $K3$ surfaces we obtain
\begin{equation}
\label{eqn:hol}
1+\zeta^{-m} = \sum_{i,j} \frac{a_{ij}}{(1-\zeta^i)(1-\zeta^j)} + 
\sum_{C\subseteq X^\sigma} (1-g(C)) \frac{1+\zeta^n}{(1-\zeta^n)^2},
\end{equation}
where 
\begin{itemize} 
\item 
$a_{ij}$ is the number of 
$\sigma$-fixed points $\mathfrak p$ with weights $\beta_{\mathfrak p}=(i,j)$ in 
the tangent bundle 
at $\mathfrak p$, 
$$
i+j\equiv n \pmod N, \quad i,j\neq 0,
$$
\item
$C\subseteq X^\sigma$ are  (smooth irreducible) curves,
\end{itemize}
(see \cite{NikFinite} or \cite[Lemma 1.1]{ACV}).

Formula \eqref{eqn:hol} immediately implies:

\begin{lemm}
\label{lem:nonz}
Let $N=n\cdot m$ with $m\neq 2$. Then 
$$
X^\sigma \neq \emptyset.
$$
\end{lemm}

\begin{proof}
Consider equation (\ref{eqn:hol}). If $m\neq 2$ then the left-hand side is nonzero. It follows
that the sums on the right-hand side are nonempty. Since these are indexed by fixed points or curves, we 
conclude that $X^{\sigma} \neq \emptyset$. 
\end{proof}

Lemma~\ref{lem:nonz} shows that we always have fixed points in the {\bf symplectic case}.
In the {\bf purely nonsymplectic case}, where $N=m$, or equivalently, $n=1$,
Lemma~\ref{lem:nonz} guarantees fixed points, except where $m=N=2$.  
In this case, the only fixed-point free action is the Enriques involution. Such an involution is characterized by the sublattice of its fixed classes
(see, e.g.,  \cite{Nik-ICM}, \cite[Th. 1.1]{AS}, \cite[Th. 3.1]{AST}):
$$
\Pic(X)^\sigma \simeq U(2) \oplus E_8(2).
$$

We turn to the {\bf mixed case} where $m,n>1$. 
Lemma~\ref{lemm:symf} guarantees that the existence of $\sigma$-fixed points is governed by the trace of $\sigma$ on $\widetilde{\rH}$, i.e., is a derived invariant. This completes the proof of Theorem~\ref{theo:main}.

\subsection{Role of classification in the proof}
\label{sect:RCP}
Despite initial expectations, the proof of Theorem~\ref{theo:main} does not hinge on classification. At the same time, the comprehensive enumeration in \cite{BH} does raise interesting questions.

\begin{quote}
Can we explicitly describe all types of cyclic automorphisms with $X^\sigma=\emptyset$?
Deeper arithmetic problems -- extensions to more complicated isotrivial families or the
$p$-adics -- would require understanding of all finite groups of automorphisms.  
\end{quote}

We present indicative examples of actions with $X^{\sigma}=\emptyset$.

Nikulin \cite{NikFinite} classified symplectic automorphisms of a K3 surface $X$ of order $n$ (in the notation above, $N=n$ and $m=1$): We necessarily have
$n \le 8$ and $X^\sigma\neq \emptyset$. 
Moreover, $X^\sigma$ is a finite set of isolated points, 
whose structure is given by
\begin{itemize}
\item{$n=2:$ $8$ fixed points}
\item{$n=3:$ $6$ fixed points}
\item{$n=4:$ $4$ fixed points (and $4$ points with order two stabilizer)}
\item{$n=5:$ $4$ fixed points}
\item{$n=6:$ $2$ fixed points (and $4$ points with order three stabilizer, and $6$
points with order two stabilizer)}
\item{$n=7:$ $3$ fixed points}
\item{$n=8:$ $2$ fixed points (and $2$ points with order four stabilizer, $4$ points with order two
stabilizer).}
\end{itemize}
Mukai \cite{MukaiM} gave a classification of all finite groups acting symplectically.

Detailed results are also available for purely 
nonsymplectic automorphisms of order $m$.
The cases of {\em prime} order 
have been considered in \cite{AST}, and various other special cases in, e.g.,  \cite{ACV}, \cite{AS}, \cite{dillies}, \cite{ST}, \cite{Taki2}. 
A complete classification, including an analysis of possible fixed point configurations,
is presented in \cite[Appendix B]{BH}:
Let $\sigma$ be a purely nonsymplectic automorphism of a K3 surface $X$ of order 
$m$. Then 
$$
m\in \{ 2,\ldots, 28\}\setminus \{ 23\},
$$
or 
$$
m\in \{ 30, 32, 33, 34, 36, 40, 44, 48, 50, 54, 66\}.
$$

We return to our general situation where $C_N, N=nm,$ acts on a K3 surface, via $m$th roots of unity on the symplectic form. Lemma~\ref{lem:nonz} allows us to restrict to $m=2$.

By \cite[Lem. 4.8]{Keum}, $m=2$ implies that $n\neq 8$. For $n=7$, the number of fixed points of the subgroup $C_7=\langle \sigma^2\rangle \subset G$ is three, thus we are guaranteed $\sigma$-fixed points. For $n\le 6$ there exist fixed-point free actions. We record: 

\begin{itemize}
\item $N=2\cdot 2$: Then $X^\sigma$ is either empty, or it consists of 2, 4, 6, or 8 points \cite[Prop. 2]{AS}; when $X^\sigma=\emptyset$, the $\sigma^*$-action on $\rH^2(X,\bQ)$ has
eigenvalues $1$ and $-1$ with multiplicities $6$ and $8$, this characterizes such actions
\cite[Prop. 2]{AS}. 
K3 surfaces with $N=2\cdot 2$ have
$\operatorname{rk} \Pic(X) \ge 14$
\cite[Rema.~1.3]{AS}.
Examples of such actions can be found in 
\cite[Exam.~1.2]{AS}. 
\item $N=3\cdot 2$: Then $X^\sigma$ is either empty, or it consists of 2, 4, or 6 points \cite[Prop. 3.4]{ST}. 
    \item $N=4\cdot 2$: Here the enumeration of cases is more complicated. The classification in \cite{BH} of
symplectic actions on K3 surfaces lists only {\em maximal} actions: If $G$ acts symplectically
on a K3 surface $X$, consider its {\em saturation}, i.e., the largest subgroup $G'\subset \Aut(X)$ such that
$\rH^2(X,\bZ)^{G'} = \rH^2(X,\bZ)^G$ -- a finite group acting symplectically on $X$. Thus the enumeration
requires checking many subgroups for the presence of an element of the prescribed order.  

Consider, for instance, the group with GAP id {\tt (8,1)} 
from the second column of Table 3 in  
\cite{BH}, which lists three types. The  
possibilities for $\chi(\sigma^r)$, for $r=1,2,4$,  are
$$
\begin{array}{c|c|c} 
\sigma & \sigma^2 & \sigma^4 \\
\hline
0 & 4 & 8 \\
2 & 4 & 8 \\
4 & 4 & 8 \\
\end{array}
$$
\item $N=5\cdot 2$: Note that $C_n,n=5,6,7$ does not appear as the saturation of a mixed action with $m=2$ 
\cite[Table~3]{BH}. However, there are larger groups admitting cyclic subgroups of order ten acting on the 
symplectic form via $\pm 1$. 

For example, suppose that $G$ is an extension
$$ 
1 \rightarrow \mathfrak{A}_6 \rightarrow G \rightarrow \mu_2 \rightarrow 1,
$$
where the alternating group is the maximal symplectic subgroup. Assume that $G$ has GAP id {\tt (720,764)},
which admits elements of order ten. (Of course, $\mathfrak{A}_6$ has no such elements!) There are
six different occurences of this group in the classification. The one with K3 id $(79.2.1.3)$ has
distinguished generator (in the nomenclature of the data sets supporting \cite{BH}) $\sigma$ of
order ten with $\chi(\sigma)=0$. 
\end{itemize}

\subsection{Relations to Burnside invariants}
Brandhorst and Hofmann \cite{BH} explore cases where the data from the fixed-point formulas are insufficient
to characterize the automorphism. These are called {\em ambiguous cases}, at least in the purely
nonsymplectic context \cite[\S 7]{BH}. 

It would be interesting to consider these from the perspective
of the Burnside group:
Given the action of a finite cyclic group $G$ on a K3 surface,
there is a combinatorial object consisting of subgroups $G_i \subset G$
indexed by strata $Z_i \subset X$ with nontrivial 
stabilizer $G_i$, labeled by the induced action on $Z_i$, and the representation
type of the action of $G_i$ on the normal bundle; data of such type are building blocks 
of equivariant Burnside groups introduced in \cite{BnG}. The tables in \cite[Appendix B]{BH} list possible configurations of fixed points and curves, 
for purely nonsymplectic actions.  
How much of the Burnside data can be extracted from the representation of $G$ on the Mukai lattice? 

The paper \cite{KT-dp} explores such a connection for actions of finite groups on del Pezzo surfaces.

\

Another interesting problem is to identify which actions classified in \cite{BH} are 
derived equivalent and even to classify finite groups of autoequivalences of
K3 surfaces \cite{HuyDer}. For example, Ouchi \cite[\S 8]{Ouchi} has
found symplectic autoequivalences of orders $9$ and $11$ via
cubic fourfolds; these cannot be realized as symplectic actions on 
K3 surfaces.

\bibliographystyle{alpha}
\bibliography{isoshort}

\begin{thebibliography}{AAHF21}

\bibitem[AAHF21]{AAFH}
Nicolas Addington, Benjamin Antieau, Katrina Honigs, and Sarah Frei.
\newblock Rational points and derived equivalence.
\newblock {\em Compos. Math.}, 157(5):1036--1050, 2021.

\bibitem[ACV20]{ACV}
Michela Artebani, Paola Comparin, and Mar\'{\i}a~Elisa Vald\'{e}s.
\newblock Order 9 automorphisms of {K}3 surfaces.
\newblock {\em Comm. Algebra}, 48(9):3661--3672, 2020.

\bibitem[ADPZ17]{ADPZ}
Kenneth Ascher, Krishna Dasaratha, Alexander Perry, and Rong Zhou.
\newblock Rational points on twisted {K}3 surfaces and derived equivalences.
\newblock In {\em Brauer groups and obstruction problems}, volume 320 of {\em
  Progr. Math.}, pages 13--28. Birkh\"{a}user/Springer, Cham, 2017.

\bibitem[AKW17]{AKW}
Benjamin Antieau, Daniel Krashen, and Matthew Ward.
\newblock Derived categories of torsors for abelian schemes.
\newblock {\em Adv. Math.}, 306:1--23, 2017.

\bibitem[AS68]{At-S-3}
M.~F. Atiyah and I.~M. Singer.
\newblock The index of elliptic operators. {III}.
\newblock {\em Ann. of Math. (2)}, 87:546--604, 1968.

\bibitem[AS15]{AS}
Michela Artebani and Alessandra Sarti.
\newblock Symmetries of order four on {K}3 surfaces.
\newblock {\em J. Math. Soc. Japan}, 67(2):503--533, 2015.

\bibitem[AST11]{AST}
Michela Artebani, Alessandra Sarti, and Shingo Taki.
\newblock {$K3$} surfaces with non-symplectic automorphisms of prime order.
\newblock {\em Math. Z.}, 268(1-2):507--533, 2011.
\newblock With an appendix by Shigeyuki Kond{\=o}.

\bibitem[BH21]{BH}
Simon Brandhorst and Tommy Hofmann.
\newblock Finite subgroups of automorphisms of {K}3 surfaces, 2021.
\newblock {\tt arXiv:2112.07715}.

\bibitem[Bue89]{Buell}
Duncan~A. Buell.
\newblock {\em Binary quadratic forms}.
\newblock Springer-Verlag, New York, 1989.
\newblock Classical theory and modern computations.

\bibitem[Dil12]{dillies}
Jimmy Dillies.
\newblock On some order 6 non-symplectic automorphisms of elliptic {K}3
  surfaces.
\newblock {\em Albanian J. Math.}, 6(2):103--114, 2012.

\bibitem[Ela11]{elagin}
A.~D. Elagin.
\newblock Cohomological descent theory for a morphism of stacks and for
  equivariant derived categories.
\newblock {\em Mat. Sb.}, 202(4):31--64, 2011.

\bibitem[HMS09]{HMS}
Daniel Huybrechts, Emanuele Macr\`\i, and Paolo Stellari.
\newblock Derived equivalences of {$K3$} surfaces and orientation.
\newblock {\em Duke Math. J.}, 149(3):461--507, 2009.

\bibitem[HS05]{HS}
Daniel Huybrechts and Paolo Stellari.
\newblock Equivalences of twisted {$K3$} surfaces.
\newblock {\em Math. Ann.}, 332(4):901--936, 2005.

\bibitem[HT17]{HT-derived}
Brendan Hassett and Yuri Tschinkel.
\newblock Rational points on {K}3 surfaces and derived equivalence.
\newblock In {\em Brauer groups and obstruction problems}, volume 320 of {\em
  Progr. Math.}, pages 87--113. Birkh\"{a}user/Springer, Cham, 2017.

\bibitem[HT22]{HT21}
Brendan Hassett and Yuri Tschinkel.
\newblock Equivariant geometry of odd-dimensional complete intersections of two
  quadrics.
\newblock {\em Pure Appl. Math. Q.}, 18(4):1555--1597, 2022.

\bibitem[HT23]{HT23}
Brendan Hassett and Yuri Tschinkel.
\newblock Involutions on {K}3 surfaces and derived equivalence, 2023.
\newblock preprint.

\bibitem[Huy06]{Hu-FM}
Daniel Huybrechts.
\newblock {\em Fourier-{M}ukai transforms in algebraic geometry}.
\newblock Oxford Mathematical Monographs. The Clarendon Press, Oxford
  University Press, Oxford, 2006.

\bibitem[Huy16a]{Hu}
Daniel Huybrechts.
\newblock {\em Lectures on {K}3 surfaces}, volume 158 of {\em Cambridge Studies
  in Advanced Mathematics}.
\newblock Cambridge University Press, Cambridge, 2016.

\bibitem[Huy16b]{HuyDer}
Daniel Huybrechts.
\newblock On derived categories of {K}3 surfaces, symplectic automorphisms and
  the {C}onway group.
\newblock In {\em Development of moduli theory---{K}yoto 2013}, volume~69 of
  {\em Adv. Stud. Pure Math.}, pages 387--405. Math. Soc. Japan, [Tokyo], 2016.

\bibitem[Keu16]{Keum}
JongHae Keum.
\newblock Orders of automorphisms of {K}3 surfaces.
\newblock {\em Adv. Math.}, 303:39--87, 2016.

\bibitem[KS15]{KS}
Andreas Krug and Pawel Sosna.
\newblock Equivalences of equivariant derived categories.
\newblock {\em J. Lond. Math. Soc. (2)}, 92(1):19--40, 2015.

\bibitem[KT22a]{KT-dp}
Andrew Kresch and Yuri Tschinkel.
\newblock Cohomology of finite subgroups of the plane {C}remona group, 2022.
\newblock {\tt arXiv:2203.01876}.

\bibitem[KT22b]{BnG}
Andrew Kresch and Yuri Tschinkel.
\newblock Equivariant birational types and {B}urnside volume.
\newblock {\em Ann. Sc. Norm. Super. Pisa Cl. Sci. (5)}, 23(2):1013--1052,
  2022.

\bibitem[LO15]{LO}
Max Lieblich and Martin Olsson.
\newblock Fourier-{M}ukai partners of {K}3 surfaces in positive characteristic.
\newblock {\em Ann. Sci. \'{E}c. Norm. Sup\'{e}r. (4)}, 48(5):1001--1033, 2015.

\bibitem[Muk87]{Mukai}
Shigeru Mukai.
\newblock On the moduli space of bundles on {$K3$} surfaces. {I}.
\newblock In {\em Vector bundles on algebraic varieties ({B}ombay, 1984)},
  volume~11 of {\em Tata Inst. Fund. Res. Stud. Math.}, pages 341--413. Tata
  Inst. Fund. Res., Bombay, 1987.

\bibitem[Muk88]{MukaiM}
Shigeru Mukai.
\newblock Finite groups of automorphisms of {$K3$} surfaces and the {M}athieu
  group.
\newblock {\em Invent. Math.}, 94(1):183--221, 1988.

\bibitem[Nik79a]{NikFinite}
V.~V. Nikulin.
\newblock Finite groups of automorphisms of {K}\"ahlerian {$K3$} surfaces.
\newblock {\em Trudy Moskov. Mat. Obshch.}, 38:75--137, 1979.

\bibitem[Nik79b]{nik-lattice}
V.~V. Nikulin.
\newblock Integer symmetric bilinear forms and some of their geometric
  applications.
\newblock {\em Izv. Akad. Nauk SSSR Ser. Mat.}, 43(1):111--177, 238, 1979.
\newblock English translation: Math USSR-Izv. {\bf 14} (1979), no. 1, 103–167
  (1980).

\bibitem[Nik87]{Nik-ICM}
V.~V. Nikulin.
\newblock Discrete reflection groups in {L}obachevsky spaces and algebraic
  surfaces.
\newblock In {\em Proceedings of the {I}nternational {C}ongress of
  {M}athematicians, {V}ol. 1, 2 ({B}erkeley, {C}alif., 1986)}, pages 654--671.
  Amer. Math. Soc., Providence, RI, 1987.

\bibitem[Orl97]{Orlov}
D.~O. Orlov.
\newblock Equivalences of derived categories and {$K3$} surfaces.
\newblock {\em J. Math. Sci. (New York)}, 84(5):1361--1381, 1997.
\newblock Algebraic geometry, 7.

\bibitem[Ouc21]{Ouchi}
Genki Ouchi.
\newblock Automorphism groups of cubic fourfolds and {K}3 categories.
\newblock {\em Algebr. Geom.}, 8(2):171--195, 2021.

\bibitem[Plo07]{ploog}
David Ploog.
\newblock Equivariant autoequivalences for finite group actions.
\newblock {\em Adv. Math.}, 216(1):62--74, 2007.

\bibitem[SD74]{SD}
B.~Saint-Donat.
\newblock Projective models of {$K3$} surfaces.
\newblock {\em Amer. J. Math.}, 96:602--639, 1974.

\bibitem[SyT21]{ST}
Nirai Shin-yashiki and Shingo Taki.
\newblock Non-purely non-symplectic automorphisms of order 6 on {$K3$}
  surfaces.
\newblock {\em Proc. Japan Acad. Ser. A Math. Sci.}, 97(8):61--66, 2021.

\bibitem[Tak10]{Taki2}
Shingo Taki.
\newblock Non-symplectic automorphisms of 3-power order on {$K3$} surfaces.
\newblock {\em Proc. Japan Acad. Ser. A Math. Sci.}, 86(8):125--130, 2010.

\end{thebibliography}

\end{document}